\documentclass{amsart}
\usepackage{amsmath,amssymb,amsthm,latexsym,amsfonts}
\usepackage{palatino,fancyhdr,graphicx}
\usepackage{color}
\usepackage{verbatim}
\usepackage{tikz-cd}
\usepackage{setspace}
\usepackage{adjustbox}

\def \CC {{\mathbb C}}
\def \RR {{\mathbb R}}
\def \QQ {{\mathbb Q}}
\def \H {{\mathbb H}}

\def \ZZ {{\mathbb Z}}

\def \p {{\mathfrak p}}
\def \P {{\mathfrak P}}

\def \Oc {\mathcal{O}}
\def \D {{\mathcal D}}
\def \Lc {\mathcal{L}}
\def \R {{\mathcal R}}

\def \uv {{\boldsymbol u}}

\def \xv {{\boldsymbol x}}
\def \yv {{\boldsymbol y}}

\def \alv {{\boldsymbol \alpha}}

\def \ord {\Lambda}

\newcommand\QA[2] {\left(\frac{#1}{#2}\right)}

\newcommand\tr[1] {{\rm{tr}}\left(#1\right)}
\newcommand\n[1] {{\rm{n}}\left(#1\right)}
\newcommand\di[1] {{\rm{disc}}\left(#1\right)}
\newcommand\de[1] {{\rm{det}}\left(#1\right)}
\newcommand\trac[2] {{\rm tr}_{#1}\left(#2\right)}
\newcommand\nrd[2] {{\rm n}_{#1}\left(#2\right)}
\newcommand\Trac[2] {{\rm Tr}_{#1}\left(#2\right)}
\newcommand\Nm[2] {{\rm N}_{#1}\left(#2\right)}
\newcommand\N[1] {{\mathcal N}(#1)}
\newcommand\md[2] {\equiv#1~{\rm mod}~#2}

\theoremstyle{definition}
\newtheorem{thm}{Theorem}[section]

\newtheorem{definition}[thm]{Definition}
\newtheorem{theorem}[thm]{Theorem}
\newtheorem{lemma}[thm]{Lemma}
\newtheorem{proposition}[thm]{Proposition}

\newtheorem{corollary}[thm]{Corollary}
\newtheorem{remark}[thm]{Remark}

\newtheorem{example}[thm]{Example}

\begin{document}

\title[Arakelov-Modular Lattices over Quaternion Algebras]{Construction of Arakelov-Modular Lattices over Totally Definite Quaternion Algebras}
\author{Xiaolu Hou}
\address{
Division of Mathematical Sciences\\
Nanyang Technological University, Singapore}
\email {HO0001LU@e.ntu.edu.sg}

\begin{abstract}
We study ideal lattices constructed from totally definite quaternion algebras over totally real number fields, 
and generalize the definition of Arakelov-modular lattices over number fields proposed in \cite{Bayer}.
In particular, we prove for the case where the totally real number field is $\QQ$, that for $\ell$ a prime integer, there always exists a totally definite quaternion over $\QQ$ from which an Arakelov-modular lattice of level $\ell$ can be constructed.
\end{abstract}
\maketitle

\vspace{1cm}
\\


%
%
%
\section{Introduction}

A lattice is a pair $(L,b)$, where $L$ is a free $\ZZ-$module of finite rank and $b$ is a positive definite symmetric $\ZZ-$bilinear form on $L\otimes_\ZZ\RR$.
If $L$ has rank $n$ and $\{v_1,v_2,\dots,v_n\}$ is a $\ZZ-$basis for $L$, then the matrix $G:=(b(v_i,v_j))_{1\leq i,j\leq n}$ is the \emph{Gram matrix} of $(L,b)$.
The determinant of the Gram matrix of $(L,b)$ is called the \emph{discriminant} of $L$.
The \emph{dual lattice} of $(L,b)$ is the pair $(L^*,b)$, where
\[
	L^*=\{x\in L\otimes_\ZZ\RR:b(x,y)\in\ZZ\ \forall y\in L\}.
\]
An \emph{integral} lattice is a lattice $(L,b)$ such that $L\subseteq L^*$.
An integral lattice is called \emph{even} if $b(x,x)\in2\ZZ$ for all $x\in\ZZ$ and \emph{odd} otherwise~\cite{Ebeling,Conway}.

For $\ell$ a positive integer, an \emph{$\ell-$modular lattice} \cite{Quebbemann} is an integral lattice such that there exists a bijective $\ZZ-$module homomorphism $\varphi:L^*\to L$ and
\[
	\ell b(x,y)=b(\varphi(x),\varphi(y))\ \forall x,y\in L^*.
\]
When $\ell=1$ we have a \emph{unimodular lattice}.

A common way of constructing $\ell-$modular lattices is by using ideals of number fields~\cite{Feit,Batut,Bayer} the resulting lattices are then called \emph{ideal lattices}~\cite{BayerIdeallattice}.
In~\cite{Feit}, a construction of unimodular lattices by ideal lattices over $\QQ(\sqrt{-3})$ is given.
A more general construction over cyclotomic extension of imaginary quadratic fields can be found in~\cite{Batut}.
The notion of Arakelov-modular lattice, which gives modular lattice, was introduced in~\cite{Bayer}, where the authors classified the construction of Arakelov-modular lattices over cyclotomic fields.
This paper generalizes this notion to construct Arakelov-modular lattices over the ideals of totally definite quaternion algebras over totally real number fields.

The construction by ideals of quaternions was also used in~\cite{Martinet} for two particular cases, $\QA{-1,-1}{\QQ}$ and $\QA{-1,-3}{\QQ}$, for constructing 2 and 3 modular lattices respectively.
This is a special case of our construction (see Example~\ref{exp:Martinet}).

We will discuss in details the definition of the bilinear form we use in Section 2 and introduce the definition of ideal lattices over totally definite quaternions in Section 3.
In Section 4 the generalized notion of Arakelov-modular lattice is introduced.
In Section 5, we will focus on the case where the underlying number field is the rational field, for which we obtain existence results and classify Arakelov-modular lattices for $\ell$ a prime.
In particular, we will prove that, given any prime $\ell$, there exists a totally definite quaternion algebra over $\QQ$ over which an Arakelov-modular lattice of level $\ell$ can be constructed (Theorem~\ref{thm:QA-lprimerationalfield}).

%
%
%
\section{Totally Definite Quaternion Algebras}

Let $K$ be a totally real number field with degree $n=[K:\QQ]$.
Let $A=\QA{a,b}{K}$ be a quaternion algebra over $K$ with standard basis $\{1,i,j,ij\}$, i.e., $A$ is a $4$-dimensional vector space over $K$ with basis $\{1,i,j,ij\}$ such that
\[
i^2=a,\ j^2=b,\ ij=-ji,
\]
for some $a,b\in K^\times$.

The quaternion algebra $A$ is equipped with a canonical involution (or conjugation) given by 
\begin{eqnarray}\label{eqn:conjugation}
^-:A&\to&A\\
\alpha=x_0+x_1i+x_2j+x_3ij&\mapsto&\bar{\alpha}=x_0-x_1i-x_2j-x_3ij,\nonumber
\end{eqnarray}
from which are defined the reduced trace on $A$:
\begin{eqnarray*}
	{\rm tr}_{A/K}:A&\to& K\\
	\alpha&\mapsto&\alpha+\bar{\alpha},
\end{eqnarray*}
and similarly the reduced norm:
\begin{eqnarray*}
	{\rm n}_{A/K}:A&\to& K\\
	\alpha&\mapsto&\alpha\bar{\alpha}.
\end{eqnarray*}

For $\alpha=x_0+x_1i+x_2j+x_3ij\in A$ 
\begin{eqnarray}\label{eqn:traceQA}
	{\rm tr}_{A/K}(\alpha)&=&\alpha+\bar{\alpha}=2x_0\\
	\nrd{A/K}{\alpha}&=&\alpha\bar{\alpha}=x_0^2-ax_1^2-bx_2^2+abx_3^2\nonumber.
\end{eqnarray}

Denote the real embeddings of $K$ by $\sigma_1,\dots,\sigma_n$.
Let $K_v$ denote the completion of $K$ at the Archimedean place corresponding to a real embedding $\sigma \in \{\sigma_1,\ldots,\sigma_n\}$, then either $A_v=A\otimes_KK_v\cong\H=\QA{-1,-1}{\mathbb{R}}$, in which case $A$ is ramified, or $A_v$ is isomorphic to the ring $M_2(\RR)$ of $2\times 2$ real coefficients~\cite[p.93]{Maclachlan}. 
Then $A$ is unramified.

Let $s_1$ be the number of real places at which $A$ is ramified. Then for a number field of signature $(r_1,r_2)$, that is with $r_1$ real embeddings and $r_2$ pairs of complex embeddings, we have the following isomorphism $\phi$:
\begin{equation}\label{eqn:AR}
	A_\RR:=A\otimes_\QQ\RR\cong\oplus s_1\H\oplus(r_1-s_1)M_2(\RR)\oplus r_2 M_2(\CC)
\end{equation}
so that in particular, for $K$ totally real 
\[
A_\RR=A\otimes_\QQ\RR\cong\oplus s_1\H\oplus(n-s_1)M_2(\RR).
\]
We will often identify an element $\alv=\alpha\otimes h \in A_\RR$ with its image $\phi(\alv)=(\alpha_1,\ldots,\alpha_n)$,
where $\alpha_m=h\sigma_m(\alpha)$, and $\sigma_m$ is understood componentwise on the coefficients $x_0,x_1,x_2,x_3$ of $\alpha$, written in a suitable basis.

The conjugation (\ref{eqn:conjugation}) on $A$ can be extended to $A_\RR$ by defining the conjugation of $\alv=\alpha \otimes h$, $h\in\RR^\times$, to be that of $\phi(\alpha \otimes h)$, and the conjugation of $\phi(\alpha \otimes h)$ to be in turn $\overline{\phi(\alpha \otimes h)}=\phi(\bar{\alpha}\otimes h)$.
We have $\overline{\phi(\alpha \otimes h)}=\phi(\alpha \otimes h) \iff \bar{\alpha}=\alpha$.
Define thus the set
\[
\mathcal{P}=\{\boldsymbol{\alpha}:\boldsymbol{\alpha}\in A_\RR,\boldsymbol{\alpha}=\bar{\boldsymbol{\alpha}}\}=\{\alpha\otimes h:\alpha\in K^\times\},
\]
which belongs to the center of $A_\RR$.

Take $\boldsymbol{\alpha}\in\mathcal{P}$, and define
\begin{eqnarray}\label{eqn:bilinearformqalpha}
q_\alv:A_\RR\times A_\RR&\to&\RR\\
(\xv,\yv)&\mapsto&\tr{\alv\xv\bar{\yv}} \nonumber
\end{eqnarray}
where ${\rm tr}$ denote the reduced trace on the separable $\RR-$algebra $A_\RR$, given by
\[
\tr{\xv}=\sum_{m=1}^{s_1}\trac{\H/\RR}{x_m}+\sum_{m=s_1+1}^{r_1}\trac{M_2(\RR)/\RR}{x_m}+\sum_{m=r_1+1}^{r_1+r_2}\trac{M_2(\CC)/\RR}{x_m}
\]
for ${\xv}=(x_1,\ldots,x_n) \in A_\RR$ and with
\[
\trac{M_2(\CC)/\RR}{x_m}=\trac{M_2(\CC)/\CC}{x_m}+\overline{\trac{M_2(\CC)/\CC}{x_m}}.
\]

\begin{lemma}
For $\boldsymbol{\alpha}\in\mathcal{P}$, $q_\alv$ is a non-degenerate symmetric $\ZZ-$bilinear form, and
\[
	q_\alv(\uv\xv,\yv)=q_\alv(\xv,\bar{\uv}\yv)
\]
for all $\xv,\yv,\uv\in A_\RR$.
\end{lemma}
\begin{proof}
	Since {\rm tr} is the reduced trace for the $\RR-$separable algebra $A_\RR$, it is a non-degenerate $\ZZ-$bilinear form. Now since $\alv$ is in the center of $A_\RR$,
\[
\trac{\H/\RR}{\alpha_m x_m\bar{y}_m}=\alpha_mx_m\bar{y}_m+y_m\bar{x}_m\alpha_m=\trac{\H/\RR}{\alpha_my_m\bar{x}_m}
\]
for $m=1,\ldots,s_1$. Thus
\begin{eqnarray*}
q_\alv(\xv,\yv)&=&\tr{\alv \xv\bar{\yv}} \\
           &=&\displaystyle\sum_{m=1}^{s_1}\trac{\H/\RR}{\alpha_m x_m\bar{y}_m}+\sum_{m=s_1+1}^{r_1}\trac{M_2(\RR)/\RR}{\alpha_m x_m\bar{y}_m} \\
           & &+\sum_{m=r_1+1}^{r_1+r_2}\trac{M_2(\CC)/\RR}{\alpha_m x_m\bar{y}_m}\\
           &= & q_\alv(\yv,\xv)
	\end{eqnarray*}
using a similar argument on $\trac{M_2(\RR)/\RR}{\alpha_mx_m\bar{y}_m}$ for $m=s_1+1,\ldots,r_1$ and on 
$\trac{M_2(\CC)/\CC}{\alpha_mx_m\bar{y}_m}$ for $m=r_1+1,\dots,r_1+r_2$.
This proves that $q_\alv$ is symmetric.

Now take any $\xv,\yv,\uv\in A_\RR$, and consider $q_\alv(\uv\xv,\yv)=\tr{\alv\uv\xv\bar{\yv}}$.
For $m=1,2,\dots,s_1$, using again that $\alv$ is in the center of $A_\RR$, 
\[
\trac{\H/\RR}{\alpha_mu_mx_m\bar{y}_m}=\trac{\H/\RR}{u_m\alpha_mx_m\bar{y}_m}=\trac{\H/\RR}{\alpha_mx_m\bar{y}_mu_m}
\]
and using a similar argument for $\trac{M_2(\RR)/\RR}{\alpha_mx_m\bar{y}_m}$, $m=s_1+1,\ldots,r_1$, and for 
$\trac{M_2(\CC)/\CC}{\alpha_mx_m\bar{y}_m}$, $m=r_1+1,\dots,r_1+r_2$, we conclude that
\[
q_\alv(\uv\xv,\yv)=q_\alv(\xv,\bar{\uv}\yv).
\]
\end{proof}
When $K$ is a totally real number field, and $A$ is a quaternion algebra ramified at all the real places, i.e., $s_1=r_1=n$, we say that $A$ is totally definite~\cite[34.1]{Reiner}.
For this case, define
\[
	\mathcal{P}_{>0}=\{\alv:\alv\in\mathcal{P}, \alpha_m>0\ \forall m\}.
\]
\begin{proposition}\label{prop:bilinearformQA}
	Let $K$ be a totally real number field, $A$ a totally definite quaternion algebra over $K$, and take $\alv\in\mathcal{P}_{>0}$, then 
	\begin{eqnarray*}
q_\alv:A_\RR&\to& A_\RR\\
(\xv,\yv)&\mapsto&\tr{\alv\xv\bar{\yv}}
	\end{eqnarray*}
	is a positive definite symmetric $\ZZ-$bilinear form.
\end{proposition}
\begin{proof}
We are left to prove that $q_\alv$ is positive definite.
For any $\xv\in A_\RR$, 
\begin{eqnarray*}
	\tr{\alv\xv\bar{\xv}}&=&\displaystyle\sum_{m=1}^{n}\trac{\H/\RR}{\alpha_mx_m\bar{x}_m},
\end{eqnarray*}
where
\begin{eqnarray*}
\trac{\H/\RR}{\alpha_mx_m\bar{x}_m} = \alpha_m\trac{\H/\RR}{\nrd{\H/\RR}{x_m}}=2\alpha_m\nrd{\H/\RR}{x_m} \geq 0
\end{eqnarray*}
since $A$ is totally definite and $\alv\in\mathcal{P}_{>0}$.
\end{proof}

The reduced trace of $\xv\in A_\RR$ for a totally definite quaternion algebra $A$
\[
\tr{\xv}=\sum_{m=1}^n\trac{\H/\RR}{x_m}=\sum_{m=1}^n(x_m+\bar{x}_m)
\]
is alternatively simplified, for all $x\in A$, to 
\begin{equation}\label{eqn:traceAinAQdef}
\tr{x\otimes1}=\sum_{m=1}^n(\sigma_m(x)+\overline{\sigma_m(x)})=\Trac{K/\QQ}{\trac{A/K}{x}}.
\end{equation}

Similarly, the reduced norm of $\xv\in A_\RR$~\cite[p.121 9.23]{Reiner} is
\begin{equation}
	\n{\xv}=\prod_{m=1}^n\nrd{\H/\RR}{x_m}=\prod_{m=1}^n(x_m\bar{x}_m),
\end{equation}
and for any $x\in A$~\cite[p.122 Theorem 9.27 and p.121 (9.23)]{Reiner},
\begin{equation}\label{eqn:normtdAinQA}
	\n{x\otimes1}=\nrd{A/\QQ}{x}=\Nm{K/\QQ}{\nrd{A/K}{x}}.
\end{equation}
We will write $\tr{x}$, $\n{x}$ instead of $\tr{x\otimes 1}$, $\n{x\otimes1}$ respectively whenever there is no confusion.

%
%
%
\section{Ideal Lattices in Totally Definite Quaternion Algebras}

As before, $K$ is a totally real number field of degree $n$, and $A=\QA{a,b}{K}$ is a totally definite quaternion algebra over $K$ with standard basis $\{1,i,j,ij\}$.
Let $\Oc_K$ be the ring of integers of $K$.

Let $\ord$ be an order of $A$, that is, an ideal of $A$ (i.e., a finitely generated $\Oc_K-$module contained in $A$ such that $I\otimes_{\Oc_K}K\cong A$) which is also a subring of $A$. We furthermore assume that $\ord$ is maximal (that is, not properly contained in another order).

For a maximal order $\ord$ of $A$, we define the following three sets of ideals of $A$~\cite[Section 6.7]{Maclachlan}
\begin{equation}\label{eqn:LRLambda}
	\Lc(\ord)=\{I:\Oc_\ell(I)=\ord\},\ \R(\ord)=\{I:\Oc_r(I)=\ord\},\ \ \Lc\R(\ord)=\Lc(\ord)\cap\R(\ord),
\end{equation}
where
\[
	\Oc_\ell(I)=\{\alpha\in A:\alpha I\subset I\},\ \ \Oc_r(I)=\{\alpha\in A:I\alpha\subset I\}
\]
are respectively the order on the left of $I$ and the order on the right of $I$~\cite[p.84]{Maclachlan}.

Recall the definition of the codifferent of $\ord$ over $\Oc_K$ or $\ZZ$~\cite[p.217]{Reiner}, or of $\Oc_K$ over $\ZZ$~\cite[p.159]{Neukirch}, given respectively by
\begin{eqnarray}
\D_{\ord/\Oc_K}^{-1}&=&\{x\in A:\trac{A/K}{xy}\in\Oc_K\ \forall y\in\ord\}  \in\Lc\R(\ord) \label{eqn:differentordok} \\
\D_{\ord/\ZZ}^{-1}&=&\{x\in A:\Trac{K/\QQ}{\trac{A/K}{xy}}\in\ZZ\ \forall y\in\ord\}\\
\D_{\Oc_K/\ZZ}^{-1}&=&\{x\in K:\Trac{K/\QQ}{xy}\in\ZZ\ \forall y\in\Oc_K\} \label{eqn:differentokzQA}.
\end{eqnarray}
To each corresponds an inverse ideal called different, respectively denoted by $\D_{\ord/\Oc_K} \in\Lc\R(\ord)$, $\D_{\ord/\ZZ}$ and $\D_{\Oc_K/\ZZ}$. 

\begin{definition}	
	An ideal $I\subset A$ is called a {\em generalized two-sided ideal} of a maximal order $\ord$ if there exist $t\in A^\times$ and $J\in\Lc\R(\ord)$ such that $I=Jt=\{yt:y\in J\}$.	
\end{definition}

\begin{definition}\label{def:normIinAR}
For $I=Jt$ a generalized two-sided ideal of $\ord$, its {\em reduced norm} $\n{I}$ in $A_\RR$ is by definition
\[
\n{I} = \Nm{K/\QQ}{\nrd{A/K}{J}}\n{t},
\]
where $\nrd{A/K}{J}$ is the fractional ideal generated by the elements $\{\nrd{A/K}{x}:x\in J\}$~\cite[p.199]{Maclachlan}.
\end{definition}

Consider the following symmetric positive definite $\ZZ-$bilinear form:
\begin{eqnarray}\label{eqn:bilinearformQA}
q_\alpha:A_\RR\times A_\RR&\to&\RR\nonumber\\
(x,y)&\mapsto&\tr{(\alpha\otimes1)x\bar{y}},
\end{eqnarray}
where $\alpha\in K^\times$ is totally positive and for simplicity, we write $x,y$ instead of $\xv,\yv$ if there is no confusion. This is a particular case of the previous section, where we restrict to the case when $\alv=\alpha\otimes 1$ for $\alpha\in K^\times$.
Note that $\alpha\otimes1\in\mathcal{P}_{>0}$ if and only if $\alpha$ is totally positive, i.e., $\sigma_i(\alpha)>0$ for all $\sigma_i:K\hookrightarrow\RR$.

\begin{definition}
	An {\em ideal lattice} over a maximal order $\ord$ is a pair $(I,q_\alpha)$, where $I=Jt$ is a generalized two-sided ideal of $\ord$.
\end{definition}

Take an ideal lattice $(I,q_\alpha)$ over a maximal order $\ord$ of $A$, where $I=J t$ for some $t\in A^\times$ and $J\in\Lc\R(\ord)$ is such that it admits a free $\Oc_K-$basis $\{v_1,v_2,v_3,v_4\}$.
Let $\{\beta_1,\dots,\beta_n\}$ be a $\ZZ-$basis for $\Oc_K$.
Thus $\{\beta_iv_j\}_{{\tiny \substack{1\leq i\leq n\\ 1\leq j\leq4}}}$ is a $\ZZ-$basis for $\ord$ and $\{\beta_iv_jt\}_{{\tiny \substack{1\leq i\leq n\\ 1\leq j\leq4}}}$ is a $\ZZ-$basis for $I$.

A Gram matrix of $(I,q_\alpha)$ is given by
\[
	G=(q_\alpha(\beta_{k}v_{i}t,\beta_{m}v_{j}t))_{{\tiny \substack{1\leq k,m\leq n\\ 1\leq i,j\leq4}}}.
\]	
For fixed $i,j$, $(q_\alpha(\beta_{k}v_{i}t,\beta_{m}v_{j}t))_{{\tiny \substack{1\leq k,m\leq n}}}$ is an $n\times n$ matrix whose coefficients are given by
\[
	q_\alpha(\beta_kv_it,\beta_mv_jt)=\tr{\alpha\beta_kv_it\overline{(\beta_mv_jt)}}=\tr{\alpha\beta_kv_it\bar{t}\overline{\beta_mv_j}},
\]
where $\alpha$ (identified with $\alpha\otimes1$), $\beta_k, \beta_m, t\bar{t}=\nrd{A/K}{t}\in K^\times$, so that, together with Eq. (\ref{eqn:traceAinAQdef}), we have
\begin{eqnarray*}
	q_\alpha(\beta_kv_it,\beta_mv_jt)
&=&\tr{\alpha\nrd{A/K}{t}\beta_k \beta_m v_i\bar{v}_j}\\
&=&\Trac{K/\QQ}{\trac{A/K}{\alpha\nrd{A/K}{t}\beta_k\beta_mv_i\bar{v}_j}}\\
&=&\sum_{\ell=1}^n\sigma_\ell\left(\alpha\nrd{A/K}{t}\beta_k\beta_m\trac{A/K}{v_i\bar{v}_j}\right)\\
&=&\sum_{\ell=1}^n\sigma_\ell(\alpha\nrd{A/K}{t})\sigma_\ell(\beta_k\beta_m)\sigma_\ell(\trac{A/K}{v_i\bar{v}_j}).
\end{eqnarray*}
Let $B=(\sigma_\ell(\beta_k))$, then for fixed $i,j$, $(q_\alpha(\beta_{k}v_{i}t,\beta_{m}v_{j}t))_{{\tiny \substack{1\leq k,m\leq n}}}=BH_{ij}B^\top$, where
\[
	H_{ij}=
	\begin{bmatrix}
		\sigma_1(\alpha\nrd{A/K}{n}\trac{A/K}{v_i\bar{v}_j})&\dots&0\\
		\vdots&\ddots&\vdots\\
		0&\dots&\sigma_n(\alpha\nrd{A/K}{n}\trac{A/K}{v_i\bar{v}_j})
	\end{bmatrix}.
\]
Hence
\begin{equation}\label{eqn:grammatrixQA}
G=
\begin{bmatrix}
B&0&0&0\\
0&B&0&0\\
0&0&B&0\\
0&0&0&B
\end{bmatrix}
\begin{bmatrix}
	H_{11}&H_{12}&H_{13}&H_{14}\\
	H_{21}&H_{22}&H_{23}&H_{24}\\
	H_{31}&H_{32}&H_{33}&H_{34}\\
	H_{41}&H_{42}&H_{43}&H_{44}
\end{bmatrix}
\begin{bmatrix}
B^\top&0&0&0\\
0&B^\top&0&0\\
0&0&B^\top&0\\
0&0&0&B^\top
\end{bmatrix}.
\end{equation}

\begin{proposition}\label{prop:latticediQA}
	An ideal lattice $(I,q_\alpha)$ over a maximal order $\ord$ of $A$, such that $I$ has a $\ZZ$-basis, has dimension $4n$, Gram matrix (\ref{eqn:grammatrixQA}) and discriminant 
\[
\n{\alpha}^2\n{I}^4\n{\D_{\ord/\ZZ}}^2,
\]
where $\n{\alpha}=\Nm{K/\QQ}{\nrd{A/K}{\alpha}}$ is the reduced norm of $\alpha$ in $A_\RR/\RR$, and $\n{I}$ is the reduced norm of $I$ in $A_\RR/\RR$ (see \ref{def:normIinAR}) and $\D_{\ord/\ZZ}$ is the different of $\ord$ over $\ZZ$.
\end{proposition}
\begin{proof}
We are left to compute the discriminant of $(I,q_\alpha)$, which is the determinant of $G$:
\[
\de{G}=(\de{BB^\top})^4\de{H}=\de{B}^8\de{H}, 
\]
where $H=(H_{ij})$.
After row and column permutations of $H$, we get
\begin{eqnarray*}
\de{H}  &=&\prod_{\ell=1}^n\de{\sigma_\ell(\alpha\nrd{A/K}{t}\trac{A/K}{v_i\bar{v}_j})_{i,j}}\\
	&=&\prod_{\ell=1}^n\sigma_\ell(\de{(\alpha\nrd{A/K}{t}\trac{A/K}{v_i\bar{v}_j})_{i,j}})\\
        &=&\prod_{\ell=1}^n\sigma_\ell((\alpha\nrd{A/K}{t})^4)\sigma_\ell(\de{(\trac{A/K}{v_i\bar{v}_j})_{i,j}})\\
	&=&(\Nm{K/\QQ}{\alpha\nrd{A/K}{t}})^4\Nm{K/\QQ}{(\de{(\trac{A/K}{v_i\bar{v}_j})_{i,j}}},
\end{eqnarray*}
while
\[
	\de{B}^2 =\de{(\sigma_i(\beta_j))}^2=\Delta_K
\]
where $\Delta_K$ is the discriminant of $K$ by definition.
Thus
\[
\det(G)=\Delta_K^4 \n{\alpha}^2 \n{t}^4\Nm{K/\QQ}{(\de{(\trac{A/K}{v_i\bar{v}_j})_{i,j}}}.
\]

If $J=\ord$, that is $I=\ord t$, then
\[
\det(G)=\Delta_K^4 \n{\alpha}^2 \n{t}^4\Nm{K/\QQ}{\di{\ord/\Oc_K}}.
\]
Indeed, since $\{v_1,v_2,v_3,v_4\}$ is a free $\Oc_K-$basis for $\ord$, the discriminant $\di{\ord/\Oc_K}$ is the principal ideal~\cite[p.205]{Maclachlan}
\begin{equation}\label{eqn:discordok}
	\de{(\trac{A/K}{v_iv_j})_{i,j}} \Oc_K,
\end{equation}
and 
\[
\de{(\trac{A/K}{v_i\bar{v}_j})_{i,j}}\Oc_K=\de{(\trac{A/K}{v_iv_j})_{i,j}}\Oc_K
\]
by noting that $\de{(\trac{A/K}{v_i\bar{v}_j})_{i,j}}=\de{(\trac{A/K}{v_iv_j})_{i,j}}\de{(a_{kj})_{k,j}}$
for $(a_{kj})_{k,j}\in M_4(\Oc_K)$ an invertible matrix such that $\bar{v}_j=\sum_{k=1}^4a_{kj}v_k$.
Then
\[
\Nm{K/\QQ}{\di{\ord/\Oc_K}}=|\Nm{K/\QQ}{\de{(\trac{A/K}{v_i\bar{v}_j})_{i,j}}}|.
\]
Note that the determinant of a positive definite matrix is always positive.

If $I=Jt$, $J\neq \ord$, take $x\in\Oc_K$ such that $Jx\subseteq\ord$, then $Jt\subseteq \ord x^{-1}t$ and the discriminant of $(I,q_\alpha)$ is given by (see \cite[p.2]{Ebeling}) 
\begin{eqnarray*}
	\di{(I,q_\alpha)}&=&\di{(\ord x^{-1}t,q_\alpha)}|\ord x^{-1}t/Jt|^2\\ 
	&=&\Delta_K^4\n{\alpha}^2\n{x^{-1}t}^4\Nm{K/\QQ}{\di{\ord/\Oc_K}}|\ord x^{-1}/J|^2
\end{eqnarray*}
where $\n{x^{-1}t}^4=\n{t}^4\Nm{K/\QQ}{x}^{-8}$ and 
\[
|\ord x^{-1}/J| = |\ord/Jx|= \Nm{K/\QQ}{\Nm{A/K}{Jx}}=\Nm{K/\QQ}{x}^4\Nm{K/\QQ}{\nrd{A/K}{J}}^2.
\]
Thus
\begin{eqnarray*}
\di{(I,q_\alpha)}
&=& \Delta_K^4\n{\alpha}^2\n{t}^4\Nm{K/\QQ}{\di{\ord/\Oc_K}}\Nm{K/\QQ}{\nrd{A/K}{J}}^4\\
&=& \Delta_K^4\n{\alpha}^2\n{I}^4\Nm{K/\QQ}{\di{\ord/\Oc_K}}
\end{eqnarray*}
and we are left to show that
\[
\Delta_K^4\Nm{K/\QQ}{\di{\ord/\Oc_K}}=\n{\D_{\ord/\ZZ}}^2.
\]
But~\cite[p.221]{Reiner}
\[
\Nm{K/\QQ}{\di{\ord/\Oc_K}}=\Nm{K/\QQ}{\nrd{A/K}{\D_{\ord/\Oc_K}}}^2=\n{\D_{\ord/\Oc_K}}^2=\frac{\n{\D_{\ord/\ZZ}}^2}{\n{\D_{\Oc_K}/\ZZ}^2}
\]
since $\D_{\ord/\ZZ}=\D_{\ord/\Oc_K}\D_{\Oc_K/\ZZ}$.
That~\cite[p.201]{Neukirch}
\[
\Nm{K/\QQ}{\D_{\Oc_K/\ZZ}}=\Delta_K
\]
completes the proof.
\end{proof}

Let $(I^*,q_\alpha)$ be the dual lattice of $(I,q_\alpha)$, that is
\[
I^*=\{x\in I\otimes_\ZZ\RR:q_\alpha(x,y)\in\ZZ\ \forall y\in I\}.
\]

\begin{proposition}\label{prop:duallatticeQA}
The dual of $(I,q_\alpha)$ is given by $(I^*,q_\alpha)$, where
\[
I^*=\alpha^{-1}\D_{\Oc_K/\ZZ}^{-1}\D_{\ord/\Oc_K}^{-1}\bar{I}^{-1}=\alpha^{-1}\D_{\ord/\ZZ}^{-1}\bar{I}^{-1}.
\]
\end{proposition}
\begin{proof}
	Since $\{x:x\in A, \tr{xy}\in\ZZ\ \forall y\in\bar{J}\}=\D_{\ord/\ZZ}\bar{J}^{-1}$~\cite[p.217]{Reiner},
	\begin{eqnarray*}
		I^*&=&\{x:x\in I\otimes_\ZZ\RR,\ q_\alpha(x,y)\in\ZZ,\ \ \forall y\in I\}\\
		&=&\{x:\tr{\alpha x\bar{y}}\in\ZZ\ \ \forall y\in Jt\}=\{x:\tr{\alpha xy}\in\ZZ\ \ \forall y\in\bar{t}\bar{J}\}\\
		&=&\alpha^{-1}\D_{\ord/\ZZ}^{-1}\bar{J}^{-1}\bar{t}^{-1}=\alpha^{-1}\D_{\ord/\ZZ}^{-1}\bar{I}^{-1}.
	\end{eqnarray*}
	Also since $\D_{\ord/\ZZ}^{-1}\bar{J}^{-1}\in\Lc\R(\ord)$~\cite[p.217]{Reiner},
\begin{equation}\label{eq:duallatticeproof}
I^*=\D_{\ord/\ZZ}^{-1}\bar{J}^{-1}\alpha^{-1}\bar{t}^{-1}
\end{equation}
is a generalized two-sided ideal of $\ord$ and $(I^*,q_\alpha)$ is indeed an ideal lattice over $\ord$.
\end{proof}

%
%
%

\section{Arakelov-modular Lattices in Totally Definite Quaternion Algebras}
We keep the notations from previous sections.
Let $K$ be a totally real number field with degree $n$, ring of integers $\Oc_K$, and embeddings $\{\sigma_1,\dots,\sigma_n\}$.
Let $A=\QA{a,b}{K}$ be a totally definite quaternion algebra over $K$, and let 
$\ord$ be a maximal order in $A$.

Take $\alpha\in K^\times$ and let $q_\alpha$ be the positive definite $\ZZ-$bilinear form in Eq. (\ref{eqn:bilinearformQA}).
Let $I=Jt$ be a generalized two-sided ideal in $\ord$ with $J\in\Lc\R(\ord)$, $t\in A^\times$.
Then $(I,q_\alpha)$ will denote an ideal lattice over $\ord$.

We first note that $J\in\Lc\R(\ord)$ satisfies $J=\bar{J}$. Indeed, it is known~\cite[p.273]{Reiner} that the nonzero prime ideals $\p$ of $\Oc_K$ and the prime ideals $\P$ of $\ord$ are in one-to-one correspondence given by
\[
\p=\Oc_K\cap\P,\ \ \ \P|\p\ord.
\]
For a prime ideal $\P$ of $\ord$, let then $\p=\P\cap\Oc_K$.
As for any $x\in\P\cap\Oc_K$, $x=\bar{x}$, we have
\[
\bar{\P}\cap\Oc_K=\P\cap\Oc_K=\p,
\]
and it follows that $\bar{\P}=\P$.
But since $\Lc\R(\ord)=\{I:I\text{ an ideal in } A, \Oc_\ell(I)=\Oc_r(I)=\ord\}$ forms an abelian group generated by the prime ideals of $\ord$, $\bar{\P}=\P$ in turn implies $\bar{J}=J$.

Let $\ell$ denote a positive integer.
Let $\N{\ord}$ be the normalizer of $\ord$~\cite[p.199]{Maclachlan}:
\[
	\N{\ord}=\{x\in A^\times:x\ord x^{-1}=\ord\},
\]
which is a group with respect to multiplication. For any $x\in \N{\ord}$, $x\ord=\ord x\in\Lc\R(\ord)$~\cite[p.349]{Reiner}.

We generalize the notion of Arakelov-modular lattice proposed in~\cite{Bayer} for number fields to totally definite quaternion algebras.
\begin{definition}
	We call an ideal lattice $(I,q_\alpha)$ {\em Arakelov-modular of level} $\ell$ if there exists $\beta\in\N{\ord}\cap\ord$, $t\in A^\times$ such that
\[
I=I^*\beta',\hspace{1cm} \ell=\nrd{A/K}{\beta}=\beta\bar{\beta},
\]
where $\beta'=\bar{t}\beta\bar{t}^{-1}$ and $I=Jt$ for some $J\in\Lc\R(\ord)$.
\end{definition}

\begin{remark}
\begin{enumerate}
\item We have
\[
	\beta'\bar{\beta'}=\bar{t}\beta\bar{t}^{-1}t^{-1}\bar{\beta}t=\nrd{A/K}{t}^{-1}\beta\bar{\beta}\nrd{A/K}{t}=\beta\bar{\beta}=\ell,
\]
thus $\ell=\nrd{A/K}{\beta'}=\beta'\bar{\beta'}$.

\item 
An ideal lattice that is Arakelov-modular of level $\ell$ is automatically integral. Indeed,
from Eq. (\ref{eq:duallatticeproof})
\[
	I^*=\D^{-1}_{\ord/\ZZ}\bar{J}^{-1}\alpha^{-1}\bar{t}^{-1},
\]
$\D_{\ord/\ZZ}\bar{J}^{-1}\in\Lc\R(\ord)$, and the fact that $\alpha$ is in the center of $A$, we have~\cite[p.218]{Maclachlan},
\[
	\Oc_r(I^*)=\bar{t}\Oc_r(\D_{\ord/\ZZ}^{-1}\bar{J}^{-1})\bar{t}^{-1}=\bar{t}\ord\bar{t}^{-1}.
\]
Since $\beta\in\ord$, $\beta'\in\bar{t}\ord\bar{t}^{-1}$, showing that $I=I^*\beta'\subseteq I^*$. 

\item 
An  Arakelov-modular lattice $(I,q_\alpha)$ is $\ell-$modular~\cite{Quebbemann}.
Consider the ideal lattice $(I^*,q_{\ell\alpha})$ and the map
\begin{eqnarray*}
	\varphi:I^*&\to&I=I^*\beta'\\
	x&\mapsto& x\beta'.
\end{eqnarray*}
Then $\varphi$ is a $\ZZ-$module homomorphism, and for all $x,y\in I^*$,
\begin{eqnarray*}
q_{\alpha}(\varphi(x),\varphi(y))
&=&q_{\alpha}(x\beta',y\beta')=\tr{\alpha x\beta'\bar{\beta'}\bar{y}} \\
&=&\tr{\alpha x\ell\bar{y}}=\tr{\ell\alpha x\bar{y}}=q_{\ell\alpha}(x,y).
\end{eqnarray*}
\end{enumerate}
\end{remark}
\begin{lemma}\label{lem:existenceQA}
	There exists an Arakelov-modular lattice $(I,q_\alpha)$ of level $\ell$ over $\ord$ if and only if there exists $J\in\Lc\R(\ord)$, $t\in A^\times$, $\alpha\in K$ totally positive, and $\beta\in\N{\ord}\cap\ord$ such that $\ell=\beta\bar{\beta}$ and
\[
J^2=\nrd{A/K}{t}^{-1}\alpha^{-1}\D_{\ord/\ZZ}^{-1}(\beta\ord).
\]
\end{lemma}
\begin{proof}
By the above discussions, there exists an Arakelov-modular lattice of level $\ell$ if and only if there exists $\alpha\in K^\times$, totally positive, $t\in A^\times$, $\beta\in\N{\ord}\cap\ord$, $J\in\Lc\R(\ord)$ such that $\ell=\beta\bar{\beta}$ and 
\[
	Jt=I=I^*\beta'=\alpha^{-1}\D_{\ord/\ZZ}^{-1}\bar{J}^{-1}\bar{t}^{-1}\bar{t}\beta\bar{t}^{-1}=\alpha^{-1}\D_{\ord/\ZZ}^{-1}\bar{J}^{-1}\beta\bar{t}^{-1}.
\]
Furthermore, this is equivalent to
\[
	Jt\bar{t}=\alpha^{-1}\D_{\ord/\ZZ}^{-1}J^{-1}\beta,\text{ i.e., }J\nrd{A/K}{t}=\alpha^{-1}\D_{\ord/\ZZ}^{-1}J^{-1}\beta.
\]
Also, $\nrd{A/K}{t}\in K$ which is in the center of $A$, and the above equality reduces to
\begin{equation}\label{eqn:modQAtemp}
	J\beta^{-1} J=\nrd{A/K}{t}^{-1}\alpha^{-1}\D_{\ord/\ZZ}^{-1}.
\end{equation}
Note that as $\beta\in\N{\ord}$,
\[
	\Oc_r(J\beta^{-1})=\beta\Oc_r(J)\beta^{-1}=\beta\ord\beta^{-1}=\ord,
\]
the left hand side of the above equation is well-defined~\cite[p.196]{Reiner}.
As $\beta\in\N{\ord}$, $\beta^{-1}\in\N{\ord}$, so $\beta^{-1}\ord\in\Lc\R(\ord)$.
Since $J\in\Lc\R(\ord)$,
\[
	J(\beta^{-1}\ord)J\subseteq J\beta^{-1}J.
\]
On the other hand,
\[
	J\beta^{-1}J=\left\{\sum_{\text{finite sum}}x\beta^{-1}y:x,y\in J\right\}\subseteq J(\beta^{-1}\ord)J.
\]
Hence Eq. (\ref{eqn:modQAtemp}) is equivalent to
\[
	(\beta^{-1}\ord)J^2=\nrd{A/K}{t}^{-1}\alpha^{-1}\D_{\ord/\ZZ}^{-1},
\]
i.e.,
\[
	J^2=\nrd{A/K}{t}^{-1}\alpha^{-1}\D_{\ord/\ZZ}^{-1}(\beta\ord)
\]
which concludes the proof.
\end{proof}
For any ideal $T\in\Lc\R(\ord)$, $T$ has a factorization~\cite[p.193]{Maclachlan}
\[
	T=\prod_{i=1}^k\P_i^{s_i},
\]
where $\P_i$ are prime ideals of $\ord$ and we write $v_{\P_i}(T)=s_i$.
Using this notation, Lemma \ref{lem:existenceQA} becomes
\begin{lemma}\label{lem:existenceQAeven}
There exists an Arakelov-modular lattice $(I,q_\alpha)$ of level $\ell$ over $\ord$ if and only if there exists $t\in A^\times$, $\alpha\in K$ totally positive, and $\beta\in\N{\ord}\cap\ord$ such that $\ell=\beta\bar{\beta}$ and
	\[
		v_{\P}\left(\nrd{A/K}{t}^{-1}\alpha^{-1}\D_{\Oc_K/\ZZ}^{-1}\D_{\ord/\ZZ}^{-1}(\beta\ord)\right)
	\]
	is even for all prime ideal $\P$ of $\ord$.
\end{lemma}
Note that as $\nrd{A/K}{t}^{-1}\alpha^{-1}\D_{\Oc_K/\ZZ}^{-1}\D_{\ord/\ZZ}^{-1}(\beta\ord)\in\Lc\R(\ord)$, its factors are prime ideals which are in $\Lc\R(\ord)$ and hence the product of those prime ideals is still an ideal in $\Lc\R(\ord)$.
\begin{remark}\label{rm:Classify}
	\begin{enumerate}
		\item If we have two quaternion algebras over $K$, $A$ and $A'$ that both ramify at the same finite and infinite places over $K$, there exists a $K-$algebra isomorphism $\varphi:A\to A'$~\cite[p.100]{Maclachlan}.
	If we have an ideal lattice $(I,q_\alpha)$ over some maximal order $\ord$ in $A$, we can construct an ideal lattice $(I',q_{\alpha'})$ over the maximal order $\varphi(\ord)$ in $A'$ such that $(I,q_\alpha)$ and $(I',q_\alpha')$ are isomorphic.
	Conversely, if we have an ideal lattice $(I',q_{\alpha'})$ over some maximal order $\ord'$ in $A'$, we can construct an ideal lattice $(I,q_\alpha)$ over $A$ that is isomorphic to $(I',q_{\alpha'})$. 
\item Take two maximal orders $\ord$ and $\ord'$ in a quaternion algebra $A$ over $K$ that are conjugate to each other, i.e. there exists $u\in A^\times$ such that $\ord'=u\ord u^{-1}$.
	If we have an ideal lattice $(Jt,q_\alpha)$ over $\ord$, $(uJu^{-1}t,q_\alpha)$ will be an ideal lattice over $\ord'$.
	Consider the map
	\begin{eqnarray*}
		\psi:Jt&\to&uJu^{-1}t\\
		xt&\mapsto&uxu^{-1}t,
	\end{eqnarray*}
	we have
	\begin{eqnarray*}
		\trac{A/K}{\alpha uxu^{-1}t\overline{(uyu^{-1}t)}}&=&\trac{A/K}{\alpha uxu^{-1}t\bar{t}\bar{u}^{-1}\bar{y}\bar{u}}\\
		&=&\trac{A/K}{\alpha\nrd{A/K}{t}\nrd{A/K}{u}^{-1}ux\bar{y}\bar{u}}\\
		&=&\trac{A/K}{\alpha\nrd{A/K}{t}\nrd{A/K}{u}^{-1}\bar{u}ux\bar{y}}\\
		&=&\trac{A/K}{\alpha xt\overline{yt}}.
	\end{eqnarray*}
	Thus $(Jt,q_\alpha)$ and $(uJu^{-1}t,q_\alpha)$ are isomorphic.
\end{enumerate}
\end{remark}

Write $\ell=\ell_1^2\ell_2$, where $\ell_1,\ell_2\in\ZZ_{>0}$ and $\ell_2$ is square-free.
Thanks to the following proposition, we will first focus on the case when $\ell$ is square-free.
\begin{proposition}
If there exists an Arakelov-modular lattice of level $\ell_2$ over $\ord$, then there exists an Arakelov-modular lattice of level $\ell$ over $\ord$.
\end{proposition}
\begin{proof}
	Let $(Jt,q_\alpha)$ be an Arakelov-modular lattice of level $\ell_2$ over $\ord$. 
	By Lemma \ref{lem:existenceQA}, there exists $J\in\Lc\R(\ord)$ and $t\in A^\times$, $\alpha\in K$ totally positive, and $\beta\in\N{\ord}\cap\ord$ such that $\ell_2=\beta\bar{\beta}$ and
	\[
		J^2=\nrd{A/K}{t}^{-1}\alpha^{-1}\D_{\ord/\ZZ}^{-1}(\beta\ord).
	\]
	Let $\tilde{\beta}=\ell_1\beta$, then $\ell=\tilde{\beta}\bar{\tilde{\beta}}$, $\ell_1\beta\in\N{\ord}\cap\ord$ and
\[
	J^2=\nrd{A/K}{t}^{-1}(\ell_1\alpha)^{-1}\D_{\ord/\ZZ}^{-1}(\ell_1\beta\ord)=\nrd{A/K}{t}^{-1}(\ell_1\alpha)^{-1}\D_{\ord/\ZZ}^{-1}(\tilde{\beta}\ord).
\]
As $\ell_1\in\ZZ$, $\ell_1\alpha\in K$ is totally positive.
By Lemma \ref{lem:existenceQA} again, $(Jt,q_{\ell_1\alpha})$ is an Arakelov-modular lattice of level $\ell$.
\end{proof}
So from now on, we consider $\ell$ to be a square-free positive integer unless otherwise stated.

%
%
%
\subsection{Galois extensions}
For the rest of the paper, we suppose that $K$ is a totally real number field which is Galois with Galois group $G$.

For $p\in\ZZ$ a prime, we write $\p|p$ to denote that $\p$ is a prime ideal in $\Oc_K$ above $p$.
Similarly, for $\p$ a prime ideal in $\Oc_K$, we write $\P|\p$ to denote that $\P$ is the prime ideal of $\ord$ 
such that~\cite[p.273]{Reiner} $\p=\Oc_K\cap\P,\ \ \ \P|\p\ord$.
Let Ram$(A)$, Ram$_\infty(A)$ and Ram$_f(A)$ denote the set of places, finite places, and infinite places respectively, at which $A$ is ramified.

Suppose $\ell=\prod_{i=1}^k p_i$, where $p_i\in\ZZ$ are prime numbers.
Then
\[
	\ell\Oc_K=\prod_{i=1}^k\left(\prod_{j=1}^{g_i} \p_{ij}^{e_{p_i}}\right),
	~p_i\Oc_K=\prod_{j=1}^{g_i}\p_{ij}^{e_{p_i}}.
\]

We have~\cite[p.194]{Reiner}
\begin{equation}\label{eqn:llam}
	\ell\ord=\prod_{i=1}^k\left(\prod_{j=1}^{g_i}\P_{ij}^{e_{p_i}m_{\p_{ij}}}\right),
\end{equation}
where $\P_{ij}$ is the prime ideal above $\p_{ij}$ in $\ord$ and $m_{\p_{ij}}$ is the \emph{local index}~\cite[p.270]{Reiner} of $A$ at $\p_{ij}$, which takes value $2$ if $A$ ramifies at $\p_{ij}$, and $1$ otherwise.
Assume there exists $\beta\in\N{\ord}\cap\ord$ that satisfies $\ell=\beta\bar{\beta}$.
As $\beta\ord=\ord\beta\in\Lc\R(\ord)$, $\beta\ord=\overline{\beta\ord}=\ord\bar{\beta}$, so
\[
	(\beta\ord)^2=\ord\bar{\beta}\beta\ord=\ell\ord,
\]
which gives
\begin{equation}\label{eqn:betalam}
	\beta\ord=\prod_{i=1}^k\left(\prod_{j=1}^{g_i}\P_{ij}^{\frac{e_{p_i}m_{\p_{ij}}}{2}}\right).
\end{equation}

\begin{remark}
1. If $e_{p_i}$ is odd, then $\p_{ij}\in\text{Ram}_f(A)$ for all $j$, i.e., $\forall\p_{ij}|p_i$, $\p_{ij}$ is ramified.

2. Moreover, for any prime ideal $\P$ of $\ord$
\[
	v_{\P}(\beta\ord)=\frac{1}{2}v_\P(\ell\ord)=\frac{1}{2}v_{\P}(p\ord)=\frac{e_pm_\p}{2},
\]
where $\p=\P\cap\Oc_K$, $p=\p\cap\ZZ$.
\end{remark}
Now consider $p\in\ZZ$ such that there exists $\p|p$ which is ramified, i.e. $m_\p=2$.
Then for $\P|\p$, 
\[
	v_\P(\nrd{A/K}{t}^{-1}\alpha^{-1}\D_{\Oc_K/\ZZ}^{-1})
\]	
is even and $v_{\P}(\D_{\ord/\Oc_K}^{-1})=v_{\P}(( \prod_{\p\in\text{Ram}_f(A)}\P )^{-1}) =-1$~\cite[p.273]{Reiner}.
Thus to have an Arakelov-modular lattice, by Lemma \ref{lem:existenceQAeven}, we must have $p|\ell$ and $v_\P(\beta)=\frac{e_{p}m_{\p}}{2}=e_p$ is odd. Then by the above remark, for all $\p|p$, $\p\in\text{Ram}_f(A)$.
Define
\begin{eqnarray*}
	S_{\text{Ram}}&:=&\{p\in\ZZ|\text{ there exists }\ \p\text{ above }p\text{ such that }\p\in\text{Ram}_f(A)\}.\\
	\Omega(K)&:=&\{p\in\ZZ|\ p\text{ is a prime that ramifies in }K/\QQ\}.\\
	\Omega'(K)&:=&\{p\in\Omega(K)\text{ and }e_p\text{ is even}\}.
\end{eqnarray*}

To summarize:
\begin{lemma}\label{lem:existenceQAoif}
If there exists an Arakelov-modular lattice of level $\ell$ over $\ord$, then 	
	\[
		\ell = \prod_{p\in S_{\text{Ram}}}p\prod_{p\in\Omega''(K)}p,
	\]
where $\Omega''(K)$ is a subset of $\Omega'(K)$.

Moreover, for all $p\in S_{\text{Ram}}$, the two conditions

1. $e_p$ is odd, i.e. $S_{\text{Ram}}\cap\Omega'(K)=\emptyset$;

2. $\forall\p|p$, $\p\in\text{Ram}_f(A)$,

are equivalent to, for all $\p\in\text{Ram}_f(A)$,

a. $e(\p|p)$ is odd, where $p=\p\cap\ZZ$;

b. $\sigma(\p)\in\text{Ram}_f(A)$ for all $\sigma\in G$, the Galois group of $K/\QQ$.
\end{lemma}

\begin{remark}\label{rm:existenceQAoif}
1. Note that the above Lemma implies that if there exists an Arakelov-modular lattice of level $\ell$ over $\ord$, then we must have $\di{A}|\ell\Oc_K$, for 
\[
\di{A}=\prod_{\p\in\text{Ram}_f(A)}\p
\]
the reduced discriminant of $A$~\cite[p.99]{Maclachlan}.

2. For a totally real Galois field $K$, a quaternion algebra $A$ over $K$, a maximal order $\ord$ of $A$ and a positive integer $\ell$ satisfying the conditions in the above lemma, we have
	\[
		\beta\ord=\prod_{p\in S_{\text{Ram}}}\left(\prod_{\p_i|p}\P_i\right)^{e_p}\prod_{p\notin S_{\text{Ram}},p|\ell}\left(\prod_{\p_i|p}\P_i\right)^{\frac{e_p}{2}},\ \ 
		\D_{\ord/\Oc_K}^{-1}=\prod_{p\in S_{\text{Ram}}}\left(\prod_{\p_i|p}\P_i\right)^{-1},
	\]
	where $\P_i|\p_i$ and $\p_i|p$.
	Then
	\[
		\D_{\ord/\Oc_K}^{-1}(\beta\ord)=\prod_{p\in S_{\text{Ram}}}\left(\prod_{\p_i|p}\P_i\right)^{e_p-1}\prod_{p\notin S_{\text{Ram}},p|\ell}\left(\prod_{\p_i|p}\P_i\right)^{\frac{e_p}{2}}.
	\]
\end{remark}

\subsection{Galois extensions of odd degree}

A direct corollary of Lemma \ref{lem:existenceQAeven} is obtained when $K$ is of odd degree.
\begin{corollary}\label{cor:existenceQAeven}
	If $n=[K:\QQ]$ is odd and there exists an Arakelov-modular lattice of level $\ell$ over $\ord$, then
	\[
		\ell=\prod_{p\in S_{\text{Ram}}}p.
	\]
\end{corollary}

\begin{proof}
	Since $n=[K:\QQ]$ is odd, by Lemma \ref{lem:existenceQAeven}, if $p\notin S_{\text{Ram}}$ then $p\nmid\ell$.
	Using Lemma \ref{lem:existenceQAeven} again we have
	\[
		\ell=\prod_{p\in S_{\text{Ram}}}p.
	\]
\end{proof}
\begin{remark}\label{rm:existenceQAoifodddegree}
	If $\ell=\prod_{p\in S_{\text{Ram}}}p$, by Remark \ref{rm:existenceQAoif},
	\[
\D_{\ord/\Oc_K}^{-1}(\beta\ord)=\prod_{p\in S_{\text{Ram}}}\left(\prod_{\p_i|p}\P_i\right)^{e_p-1}.
	\]
	By Lemma \ref{lem:existenceQAoif} $e_p$ is odd, then $v_\P(\D_{\ord/\Oc_K}^{-1}(\beta\ord))$ is even for any $\P$ a prime ideal in $\ord$.
\end{remark}

%
%
%
\section{Totally Definite Quaternion Algebras over $K=\QQ$}
Let $A=\QA{a,b}{\QQ}$ be a totally definite quaternion algebra over $\QQ$.
$\QQ_p$ will denote the completion of $\QQ$ at the non-Archimedean evaluation corresponding to the prime integer $p$ \cite{Neukirch}.
As there is only one infinite place, the identity, and~\cite[p.93]{Maclachlan}
\[
	\QA{a,b}{\QQ}\otimes_\QQ\RR\cong\QA{a,b}{\RR},
\]
$A$ is totally definite iff $a<0$ and $b<0$~\cite[p.92]{Maclachlan}.
Note that since the cardinality of Ram$(A)$ is even~\cite[p.99]{Maclachlan}, there are an odd number of finite places where $A$ is ramified at, i.e., Ram$_f(A)$ has odd cardinality.
Moreover, $S_{\text{Ram}}=\text{Ram}_f(A)$ for $K=\QQ$.

\begin{proposition}\label{prop:QAoverQ}
	There exists an Arakelov-modular lattice of level $\ell$ over $\ord$ if and only if $\ell=\displaystyle\prod_{p\in\text{Ram}_f(A)}p$ and there exists $\beta\in\N{\ord}\cap\ord$ such that $\ell=\beta\bar{\beta}$.
\end{proposition}
\begin{proof}
	Take $\ell=\prod_{p\in\text{Ram}_f(A)}p$ and $\beta\in\N{\ord}\cap\ord$ such that $\ell=\beta\bar{\beta}$.
	By Remark \ref{rm:existenceQAoifodddegree},
\[
	\D_{\ord/\ZZ}^{-1}(\beta\ord)=\prod_{p\in\text{Ram}_f(A)}\P^{1-1}=\ord.
\]
By Lemma \ref{lem:existenceQA}, $(\ord,q_{1})$ is an Arakelov-modular lattice of level $\ell$. 

By Corollary \ref{cor:existenceQAeven} and Lemma \ref{lem:existenceQAoif}, the proof is completed.
\end{proof}

\subsection{Existence and classification for $\ell$ prime.}
Now consider $\ell$ being a prime integer.
Our goal is to derive existence results and classify Arakelov-modular lattices for primes $\ell$.

The above proposition and Remark \ref{rm:Classify} show that for each $\ell$, it suffices to consider one quaternion $A$ that ramifies at only $\ell$.
Since we are looking at quaternions over the rational field, all maximal orders in $A$ are conjugate to each other~\cite[p.211]{Maclachlan}.
By Remark \ref{rm:Classify} again, for each quaternion $A$ we are analyzing, it suffices to consider just one maximal order $\ord$ in $A$.
Moreover, we have
\begin{proposition}
	Take $\ell$ a prime integer, $A$ a quaternion algebra over $\QQ$ that ramifies only at $\ell$ and $\ord$ a maximal order of $A$.
	Any Arakelov-modular lattice of level $\ell$ over $\ord$ is isomorphic to the lattice $(\ord,q_1)$, which is an even lattice with minimum $2$ and dimension $4$.
\end{proposition}
\begin{proof}
	Fix a quaternion algebra $A=\QA{a,b}{\QQ}$ that ramifies at only $\ell$ and a maximal order $\ord$.
	By the proof of Proposition \ref{prop:QAoverQ}, $(\ord,q_1)$ is an Arakelov-modular lattice of level $\ell$.
	For any $x\in\ord$, $q_1(x)=\trac{A/\QQ}{x\bar{x}}=\trac{A/\QQ}{\nrd{A/\QQ}{x}}$.
	As $\nrd{A/\QQ}{x}\in\ZZ$, $q_1(x)\in2\ZZ$.
	Hence $(\ord,q_1)$ is even.
	Moreover, since $q_1(1)=2$, $(\ord,q_1)$ has minimum $2$.
	Now take any Arakelov-modular lattice $(Jt,q_\alpha)$ over $\ord$ of level $\ell$.
	By Lemma \ref{lem:existenceQA} and the proof of Proposition \ref{prop:QAoverQ}, the following equation holds:
	\[
		J^2 = \nrd{A/\QQ}{t}^{-1}\alpha^{-1}\D_{\ord/\ZZ}^{-1}(\beta\ord)=\nrd{A/\QQ}{t}^{-1}\alpha^{-1}\ord.
	\]
	As $\alpha,\nrd{A/\QQ}{t}\in\QQ$, let $\gamma\ell^m=(\alpha\nrd{A\QQ}{t})^{-1}$, where $m$ is an integer and $\gamma\in\QQ$.
	For any prime $p$ and $\P$, the prime ideal above $p$ in $\ord$, we have $v_\P(\gamma\ell^m\ord)$ is even.
	If $p\neq\ell$ and $p$ divides the numerator or denominator of $\gamma$, as $m_p=1$, we must have the exopnent of $p$ in the factorization of $\gamma$ is even.
	In particular, this implies $\sqrt{\gamma}\in\QQ$.
	If $p=\ell$, then $p=\P^2$ with $\P=\beta\ord$.
	Thus we have $J=\sqrt{\gamma}\beta^m\ord$.
	As $\sqrt{\gamma}\in\QQ$ and $\beta\in\N{\ord}$, $\sqrt{\gamma}\beta^m\in\N{\ord}$.
	So $J=\sqrt{\gamma}\beta^m\ord=\ord\sqrt{\gamma}\beta^m$.
	Then the lattice $(Jt,q_\alpha)=(\ord\sqrt{\gamma}\beta^mt,q_\alpha)$.
	Define
	\begin{eqnarray*}
		h:\ord&\to&\ord\sqrt{\gamma}\beta^mt\\
		x&\mapsto&x\sqrt{\gamma}\beta^mt.
	\end{eqnarray*}
	$h$ is a $\ZZ-$module homomorphism and $\forall x,y\in\ord$
	\begin{eqnarray*}
		q_\alpha(h(x),h(y))&=&q_\alpha(x\sqrt{\gamma}\beta^mt,y\sqrt{\gamma}\beta^mt)=\trac{A/\QQ}{\alpha x\sqrt{\gamma}\beta^mt\bar{t}\bar{\beta}^m\sqrt{\gamma}\bar{y}}\\
		&=&\trac{A/\QQ}{\alpha\nrd{A/\QQ}{t}\gamma\ell^m x\bar{y}}=\trac{A/\QQ}{x\bar{y}}=q_1(x,y).
	\end{eqnarray*}
	Thus $(Jt,q_\alpha)$ is isomorphic to $(\ord,q_1)$.
\end{proof}
Recall that the \emph{Hilbert symbol} $(a,b)_p$ (or $(a,b)_v$ for $v$ corresponding to an infinite place) is defined to be $-1$ if $A$ is ramified at $p$ (or $v$) and $1$ otherwise.
Then for finite prime $p$, $(a,b)_p=-1$ iff $A\otimes_\QQ\QQ_p$ is the unique division algebra over $\QQ_p$~\cite[p.87]{Maclachlan}.
Hence we will be considering quaternion algebras $A=\QA{a,b}{\QQ}$ with $a<0,b<0$ such that $(a,b)_\ell=-1$ and $(a,b)_p=1$ for all prime $p\neq\ell$.

Let $p\neq2$ be a prime integer and let $a,b,c,x,y\in\QQ^\times$, we have~\cite{Vigneras}
\begin{enumerate}
\item $(ax^2,by^2)_p=(a,b)_p$;
\item $(a,b)_p(a,c)_p=(a,bc)_p$;
\item $(a,b)_p=(b,a)_p$;
\item $(a,1-a)_p=1$.
\end{enumerate}
The following product formula~\cite{Vigneras} holds:
	\begin{equation}\label{eqn:HSprodformula}
		\prod_{v\in\{\text{ infinite places}\}}(a,b)_v\prod_{p\in\{\text{ finite places}\}}(a,b)_p=1.
	\end{equation}
Thus, we can focus on $a,b\in\ZZ$ are $-1$ or $-p$ for $p$ is a prime.
Moreover, for $a,b\in\ZZ$ and $p\neq2$~\cite{Vigneras}
\begin{equation}\label{eqn:HSQA}
	(a,b)_p=
	\begin{cases}
		1&\text{if }p\nmid a,p\nmid b\\
		\QA{a}{p}&\text{if }p\nmid a,p||b
	\end{cases},
\end{equation}
where $\QA{a}{p}$ is the Legendre symbol, which is defined by
\[
	\QA{a}{p}=
	\begin{cases}
		1,&\text{if }a\text{ is a square }{\rm mod}~p\\
		-1,&\text{otherwise}
	\end{cases}.
\]
Recall that $A$ is ramified at the unique infinite place (identity), by the product formula (\ref{eqn:HSprodformula}),
\begin{equation}\label{eqn:HSpf}
	\prod_{p\in\{\text{ finite places}\}}(a,b)_p=-1.
\end{equation}
We have the following cases:

1. $a=-1,b=-1$ or $b=-2$, by Eq. (\ref{eqn:HSQA}), $(a,b)_p=1$ for all prime $p\neq2$.
	Then by Eq. (\ref{eqn:HSpf}), $(a,b)_2=-1$.
	Thus $\QA{-1,-1}{\QQ}$ is ramified only at $2$.

2. $a=-1,b=-p$, where $p\neq2$, then by Eq. (\ref{eqn:HSQA}), $(a,b)_q=1$ for all prime $q\neq2,p$ and 
	\[
		(a,b)_p=\QA{-1}{p}=(-1)^{\frac{p-1}{2}}=
		\begin{cases}
			-1&p\equiv3~{\rm mod}~4.\\
			1&p\equiv1~{\rm mod}~4.
		\end{cases}
	\]
	By Eq. (\ref{eqn:HSpf}), 
	\[
		(a,b)_2=\begin{cases}
			1&p\equiv3~{\rm mod}~4.\\
			-1&p\equiv1~{\rm mod}~4.
		\end{cases}
	\]
	Thus $\QA{-1,-p}{\QQ}$ is ramified only at $p$ if $p\equiv3~{\rm mod}~4$ and it is ramified only at $2$ if $p\equiv1~{\rm mod}~4$.

3. $a=-p,b=-p$, $\QA{-p,-p}{\QQ}\cong\QA{-p,-p^2}{\QQ}\cong\QA{-1,-p}{\QQ}$.

4. $a=-2,b=-p$, where $p\neq2$, then by Eq. (\ref{eqn:HSQA}), $(a,b)_q=1$ for all prime $q\neq2,p$ and 
	\[
		(a,b)_p=\QA{-2}{p}=(-1)^{\frac{p^2-1}{8}}=
		\begin{cases}
			-1&p\equiv3,5~{\rm mod}~8.\\
			1&p\equiv1,7~{\rm mod}~8.
		\end{cases}
	\]
	By Eq. (\ref{eqn:HSpf}), 
	\[
		(a,b)_2=\begin{cases}
			1&p\equiv3,5~{\rm mod}~8.\\
			-1&p\equiv1,7~{\rm mod}~8.
		\end{cases}
	\]
Thus $\QA{-2,-p}{\QQ}$ only ramifies at $p$ if $p\equiv3,5~{\rm mod}~8$ and it only ramifies at $2$ if $p\equiv1,7~{\rm mod}~8$.

6. $a=-p,b=-q$, $p\neq q\neq2$, then by Eq. (\ref{eqn:HSQA}), $(a,b)_h=1$ for all prime $h\neq2,p,q$ and
\[
	(a,b)_p=\QA{-q}{p}=\QA{-1}{p}\QA{q}{p}=(-1)^{\frac{p-1}{2}}\QA{q}{p},
\]
\[
	(a,b)_q=\QA{-p}{q}=\QA{-1}{q}\QA{p}{q}=(-1)^{\frac{q-1}{2}}\QA{p}{q}.
\]
Recall reciprocity law for Legendre symbols:
\[
	\QA{p}{q}\QA{q}{p}=(-1)^{\frac{p-1}{2}\frac{q-1}{2}}
\]
If $p,q\equiv1~{\rm mod}~4$, $(a,b)_2=-1$, $(a,b)_p=\QA{q}{p}$, $(a,b)_q=\QA{p}{q}$.
Thus $\QA{-p,-q}{\QQ}$ ramifies at only $2$ iff $\QA{p}{q}=1$.

If $p\equiv1~{\rm mod}~4$, $q\equiv3~{\rm mod}~4$, $(a,b)_2=1$, $(a,b)_p=\QA{q}{p}$, $(a,b)_q=-\QA{p}{q}$ and $\QA{p}{q}\QA{q}{p}=1$.
Thus $\QA{-p,-q}{\QQ}$ ramifies at only $p$ iff $\QA{p}{q}=-1$ and only at $q$ iff $\QA{p}{q}=1$.

If $p,q\equiv3~{\rm mod}~4$, $(a,b)_2=1$, $(a,b)_p=-\QA{q}{p}$, $(a,b)_q=-\QA{p}{q}$ and $\QA{p}{q}\QA{q}{p}=-1$.
Thus $\QA{-p,-q}{\QQ}$ ramifies at only $p$ iff $\QA{p}{q}=-1$ and only at $q$ iff $\QA{p}{q}=1$.

With the above discussion, the following lemma enables us to prove the classification result.
\begin{lemma}
	If $p\equiv1~{\rm mod}~8$ is a prime, there exists a prime $q\equiv3~{\rm mod}~4$ such that $\QA{p}{q}=-1$.
\end{lemma}
\begin{proof}
	As $p\equiv1~{\rm mod}~8$, $\QA{p}{q}=-1$ iff $\QA{q}{p}=-1$.
	Take any $c_1\in\{1,2,\dots,p-1\}$ which is a quadratic non-residue~\cite[p.84]{Hardy} of $p$.
	By Chinese Reminder Theorem~\cite[p.95]{Hardy}, there exists $c\in\{1,2,\dots,4p\}$ such that $c\equiv c_1~{\rm mod}~p$ and $c\equiv3~{\rm mod}~4$.
	Clearly, ${\rm gcd}(c,p)=1$ and ${\rm gcd}(c,4)=1$, so ${\rm gcd}(c,4p)=1$.
	By Dirichlet's Theorem~\cite[Theorem 15]{Hardy}, there are infinitely many primes of the form $4pn+c$, where $n$ denotes positive integers.
\end{proof}
Thus we have
\begin{proposition}\label{prop:QAoverQram}
Take $A$ a totally definite quaternion over $\QQ$ that ramifies at only one finite prime $p$, then we have exactly one of the following scenarios:
\begin{enumerate}
	\item[1.] $p=2, A\cong\QA{-1,-1}{\QQ}$.
	\item[2.] $p\equiv3~{\rm mod}~4, A\cong\QA{-1,-p}{\QQ}$.
	\item[3.] $p\equiv5~{\rm mod}~8, A\cong\QA{-2,-p}{\QQ}$.
	\item[4.] $p\equiv1~{\rm mod}~8$, $A\cong\QA{-p,-q}{\QQ}$, where $q\equiv3~{\rm mod}~4$ is a prime such that $\QA{p}{q}=-1$.
\end{enumerate}
\end{proposition}
Now we can classify the existence of Arakelov-modular lattices of level $\ell$ for $\ell$ a prime integer over totally definite quaternions over $\QQ$.
\begin{theorem}\label{thm:QA-lprimerationalfield}
Let $A$ be a totally definite quaternion over $\QQ$ and let $\ord$ be any maximal order of $A$.
Then there exists an Arakelov-modular lattice of level $\ell$, $\ell$ prime, over $\ord$ if and only if one of the situations is satisfied:
\begin{enumerate}
	\item[1.] $A\cong\QA{-1,-1}{\QQ}$ and $\ell=2$
	\item[2.] $A\cong\QA{-1,-\ell}{\QQ}$ and $\ell\equiv3~{\rm mod}~4$.
	\item[3.] $A\cong\QA{-2,-\ell}{\QQ}$ and $\ell\equiv5~{\rm mod}~8$.
	\item[4.] $A\cong\QA{-q,-\ell}{\QQ}$ and $\ell\equiv1~{\rm mod}~8$, where $q\equiv3~{\rm mod}~4$ is a prime such that $\QA{\ell}{q}=-1$.
\end{enumerate}
\end{theorem}
\begin{proof}
	By Propositions \ref{prop:QAoverQ} and \ref{prop:QAoverQram}, for each case it suffices to find $\ord$ and $\beta\in\N{\ord}\cap\ord$ such that $\ell=\beta\bar{\beta}$.

	As usual, let $\{1,i,j,k\}$ be a standard basis for $A=\QA{a,b}{\QQ}$, i.e. $i^2=a,j^2=b$, and $ij=k$.

	Case 1: Suppose $A=\QA{-1,-1}{\QQ}$ and $\ell=2$, take $\ord$ with basis $\{1,i,j,\frac{1+i+j+k}{2}\}$ is a maximal order of $A$~\cite[p.204]{Maclachlan}.
	Then $\beta=i-j$ satisfies $2=\beta\bar{\beta}$ and $\beta\in\ord$.
	To prove $\beta\in\N{\ord}$, it suffices to show $\beta v\beta^{-1}\in\ord$ for all $v\in\{1,i,j,\frac{1+i+j+k}{2}\}$:
	\begin{eqnarray*}
		(1-j)i(i-j)^{-1}&=&-j\in\ord\\
		(1-j)j(i-j)^{-1}&=&-i\in\ord\\
		(1-j)\frac{1+i+j+k}{2}(i-j)^{-1}&=&\frac{1-i-j-k}{2}\in\ord.
	\end{eqnarray*}

Cases 2,3,4: For $\ell\neq2$, $A\cong\QA{-q,-\ell}{\QQ}$, where
	\[
		q\begin{cases}
=1&\ell\equiv3~{\rm mod}~4\\
=2&\ell\equiv5~{\rm mod}~8\\
\equiv3~{\rm mod}~4\text{ is a prime such that }\QA{\ell}{q}=-1&\ell\equiv1~{\rm mod}~8
		\end{cases}.
	\]
	$\ZZ[1,i,j,k]$ is always an order in $A$ (see~\cite[p.84]{Maclachlan}).
	Take a maximal order $\ord\supseteq\ZZ[1,i,j,ij]$ (the existence of such a maximal order is proved in~\cite[p.84]{Maclachlan}).
	In particular, we have $j\in\ord$.
	Since $\nrd{A/\QQ}{j}=\ell$, if we prove $j\in\N{\ord}$ we are done.

	We have~\cite[p.353]{Maclachlan}
	\[
		\N{\ord}=\{x\in A^*:x\in\N{\ord_p}\ \forall p \text{ a prime integer}\}.
	\]
	If $p\neq\ell$, then $p\notin\text{Ram}_f(A)$ and~\cite[p.213]{Maclachlan}
	\[
		\N{\ord_p}=\QQ_p^*\ord_p^*.
	\]
	As $-\frac{1}{\ell}\in\ZZ_p$~\cite[p.99]{Neukirch}, $j\in\ord_p$, $\ord_p=\ZZ_p\otimes_\ZZ\ord$~\cite[p.203]{Maclachlan} gives $j^{-1}=-\frac{j}{\ell}\in\ord_p$.
	Hence $j\in\ord_p^*$ and we have $j\in\N{\ord_p}$.

	If $p=\ell$, $\N{\ord_p}=A_p^*$~\cite[p.208]{Maclachlan} and hence $j\in\N{\ord_p}$.

	We can then conclude $j\in\N{\ord}$.
\end{proof}
We also have constructive proofs for cases 2 and 3.
We need the following result~\cite[p.84,214]{Maclachlan}
\begin{enumerate}
\item $\ord$ is an order in $A$ if and only if $\ord$ is a ring of integers in $A$ which contains $\ZZ$ and is such that $\QQ\ord=A$.
\item An order $\ord$ in $A$ is maximal if and only if $\di{\ord/\ZZ}=\di{A}^2$.
\end{enumerate}
Case 2. Suppose $A\cong\QA{-1,-\ell}{\QQ}$ and $\ell\equiv3~{\rm mod}~4$, take $\ord$ with basis $\{1,i,\frac{1+j}{2},\frac{i+k}{2}\}$.
	Clearly $\ZZ\subseteq\ord$ and $\QQ\ord=A$.
	By the above, to show $\ord$ is a maximal order we need to show
	\begin{itemize}
		\item $\ord$ is a subring of $A$;
		\item the elements of $\ord$ are integers, i.e. $\trac{A/_\QQ}{x},\nrd{A/\QQ}{x}\in\ZZ$ for all $x\in\ord$;
		\item $\di{\ord/\ZZ}=\di{A}^2=\ell^2\ZZ$.
	\end{itemize}
	Since $\ord$ is a free $\ZZ-$module, $\ord$ is closed under addition.
	Also $1\in\ord$.
	To prove $\ord$ is closed under multiplication, we just need to prove the product of any two basis elements is still in $\ord$.
	Consider the following multiplication table (Table 1),
\begin{table}[ht]
\caption{Case 2 multiplication table}
{\begin{tabular}{ c | c  c  c }\hline
	$\cdot$ & $i$ & $\frac{1+j}{2}$ & $\frac{i+k}{2}$ \\\hline
	$i$ & $-1$ & $\frac{i+k}{2}$ & $\frac{-1-j}{2}$ \\
	$\frac{1+j}{2}$ & $\frac{i-k}{2}$ & $\frac{1-\ell+2j}{4}$ & $\frac{(\ell+1)i}{4}$ \\
	$\frac{i+k}{2}$ & $\frac{-1+j}{2}$ & $\frac{2k+(1-\ell)i}{4}$ & $\frac{-1-\ell}{4}$\\
\end{tabular}}
\end{table}
we have
\[
\frac{i-k}{2}=i-\frac{i+k}{2}\in\ord,\ \ \frac{-1+j}{2}=\frac{1+j}{2}-1\in\ord
\]
As $\ell\equiv3~{\rm mod}~4$, $4|(\ell+1)$, hence
\[
	\frac{1-\ell+2j}{4}=\frac{1+j}{2}-\frac{\ell+1}{4}\in\ord,\ \ \frac{(\ell+1)i}{4}\in\ord,\ \ \frac{2k+(1-\ell)i}{4}=\frac{i+k}{2}-\frac{(\ell+1)i}{4}\in\ord.
\]
We have proved that $\ord$ is closed under multiplication and hence $\ord$ is a subring of $A$.
As $\ell$ is odd, the following reduced trace table (Table 2) shows that the trace of the basis elements as well as that of the product of any two basis elements are all integers.
Each entry of the table corresponds to the reduced trace of the product of the element from the left and that from the top. For example, $(1,1)-$entry is given by $\trac{A/\QQ}{1\cdot i}=0$.
\begin{table}[ht]
\caption{Case 2 reduced trace table}
{\begin{tabular}{@{}c|ccc@{}} \hline
	$\trac{A/\QQ}{\cdot}$ & $i$ & $\frac{1+j}{2}$ & $\frac{i+k}{2}$ \\\hline
	$1$ & $0$ & $1$ & $0$\\
	$i$ & $-2$ & $0$ & $-1$ \\
	$\frac{1+j}{2}$ & $0$ & $\frac{1-\ell}{2}$ & $0$ \\
	$\frac{i+k}{2}$ & $-1$ & $0$ & $\frac{-1-\ell}{2}$\\
\end{tabular}}
\end{table}

Recall $\ell\equiv3~{\rm mod}~4$, the reduced norm table (Table 3) shows that the norm of each basis element and also that the norm of the sum of any two basis elements are integers.
Each entry of the the table here corresponds to the reduced norm of the sum of the element from the left and that from the top. For example, $(1,1)-$entry is given by $\nrd{A/\QQ}{0+i}=1$.
\begin{table}[ht]
\caption{Case 2 reduced norm table}
{\begin{tabular}{ c | c  c  c }\hline
	$\nrd{A/\QQ}{+}$ & $i$ & $\frac{1+j}{2}$ & $\frac{i+k}{2}$ \\\hline
	$0$ & $1$ & $(1+\ell)/4$ & $(\ell+1)/4$\\
	$1$ & $2$ & $(\ell+9)/4$ & $(5+\ell)/4$\\
	$i$ & $4$ & $(5+\ell)/4$ & $(9+\ell)/4$ \\
	$\frac{1+j}{2}$ &$-$ & $1+\ell$ & $(1+\ell)/2$ \\
	$\frac{i+k}{2}$ & $-$ & $-$ & $1+\ell$ \\
\end{tabular}}
\end{table}

Since the trace of the sum of two integers is an integer and the norm of the product of two integers is an integer we have proved the sum and the product of any two basis elements is still an integer in $A$.
As a subring of $A$, it follows that all the elements in $\ord$ are integers in $A$.
Hence $\ord$ is an order.

The reduced discriminant of the order $\ZZ[1,i,j,k]$ is 
\[
	\de{\begin{bmatrix}
		2&0&0&0\\
		0&-2&0&0\\
		0&0&-2\ell&0\\
		0&0&0&-2\ell
		\end{bmatrix}}\ZZ=16\ell^2\ZZ.
\]
$\ord$ is obtained from $\ZZ[1,i,j,k]$ by a basis change matrix with determinant $\frac{1}{4}$ and hence
\[
	\di{\ord}=16\ell^2\cdot\frac{1}{4^2}\ZZ=\ell^2\ZZ.
\]
We have proved $\ord$ is a maximal order.
Take $\beta=j=-1+2\cdot\frac{1+j}{2}\in\ord$, then
\[
	jij^{-1}=-i\in\ord,\ \ j\frac{i+k}{2}j^{-1}=\frac{-i-k}{2}\in\ord,\ \ j\frac{1+j}{2}j^{-1}=\frac{1+j}{2}\in\ord.
\]
This shows $j\ord j^{-1}\subseteq\ord$, since~\cite[p.349]{Reiner}
\[
	\N{\ord}=\{x\in A^\times:x\ord x^{-1}\subset\ord\},
\]
we have $\beta\in\N{\ord}$.
As $\beta\bar{\beta}=\ell$, by Proposition \ref{prop:QAoverQ}, there exists an Arakelov-modular lattice of level $\ell$ over $\ord$.

Case 3. Suppose $\ord\cong\QA{-2,-\ell}{\QQ}$ and $\ell\equiv5~{\rm mod}~8$.
In this case we take $\ord$ with basis $\{i,\frac{1+i+j}{2},j,\frac{2+i+k}{4}\}$.
As in the previous case, we consider the following three tables (Tables 4,5,6) and $\ord$ can be proved to be an order.
\begin{table}[ht]
\caption{Case 3 multiplication table}
{\begin{tabular}{@{}c|cccc@{}} \hline
	$\cdot$ & $i$ & $\frac{1+i+j}{2}$ & $j$ & $\frac{2+i+k}{4}$\\\hline
	$i$&$-2$&$\frac{i-2+k}{2}$&$k$&$\frac{i-1-j}{2}$\\
	$\frac{1+i+j}{2}$&$\frac{i-2-k}{2}$&$\frac{i+j}{2}-\frac{\ell-1}{4}$&$\frac{j+k-\ell}{2}$&$\frac{\ell+3}{8}i$\\
	$j$&$-k$&$\frac{j-k-\ell}{2}$&$-\ell$&$\frac{2j-k+\ell i}{4}$\\
	$\frac{2+i+k}{4}$&$\frac{i-1+j}{2}$&$\frac{(3-\ell)i+4j+2k}{8}$&$\frac{2j+k-\ell i}{4}$&$\frac{(1-\ell)+2i+2k}{8}$
\end{tabular}}
\end{table}
\begin{table}[ht]
\caption{Case 3 reduced trace table}
{\begin{tabular}{@{}c|cccc@{}} \hline
	$\trac{A/\QQ}{\cdot}$ & $i$ & $\frac{1+i+j}{2}$ & $j$ & $\frac{2+i+k}{4}$\\\hline
	$1$ & $0$ & $1$ & $0$ & $1$\\
	$i$ & $-4$ & $-2$ & $0$ & $-1$\\
	$\frac{1+i+j}{2}$ & $-2$ & $-\frac{\ell+1}{2}$ & $-\ell$ & $0$\\
	$j$ & $0$ & $-\ell$ & $-2\ell$ & $0$\\
	$\frac{2+i+k}{4}$ & $-1$ & $0$ & $0$ & $\frac{1-\ell}{4}$
\end{tabular}}
\end{table}
\begin{table}[ht]
\caption{Case 3 reduced norm table}
{\begin{tabular}{@{}c|cccc@{}} \hline
	$\nrd{A/\QQ}{+}$ & $i$ & $\frac{1+i+j}{2}$ & $j$ & $\frac{2+i+k}{4}$\\\hline
	$0$ & $2$ & $(\ell+3)/4$ & $\ell$ & $(\ell+3)/8$\\
	$1$ & $3$ & $(11+\ell)/4$ & $1+\ell$ & $(19+\ell)/8$\\
	$i$ & $8$ & $(\ell+19)/4$ & $\ell+2$ & $(\ell+27)/8$\\
	$\frac{1+i+j}{2}$ & - & $\ell+3$ & $(9\ell+3)/4$ & $(17+3\ell)/8$\\
	$j$ & - & - & $4\ell$ & $(9\ell+3)/8$\\
	$\frac{2+i+j}{4}$ & - & - & - & $(\ell+3)/2$
\end{tabular}}
\end{table}
Then similarly, as $\ord$ has discriminant $\ell^2\ZZ$, it is a maximal order.
By direct computation, we can prove $\beta=j\in\ord\cap\N{\ord}$.
\begin{remark}
The above computations enable us to find Arakelov-modular lattices of level $\ell$ for all $\ell\md{3}{4}$ and $\ell\md{5}{8}$.
Similar techniques can also be applied for square-free composite integers $\ell$.
\end{remark}
By Proposition \ref{prop:QAoverQ}, there exists an Arakelov-modular lattice of level $\ell$ over $\ord$.
\begin{example}\label{exp:Martinet}
\begin{enumerate}
	\item[1.] Take $A=\QA{-1,-1}{\QQ}$ and $\ord$ with basis $\{1,i,j,\frac{1+i+j+k}{2}\}$, $(\ord,q_1)$ is a $2-$modular lattice.
	\item[2.] Take $A=\QA{-1,-3}{\QQ}$ and $\ord$ with basis $\{1,i,\frac{1+j}{2},\frac{i+k}{2}\}$, $(\ord,q_1)$ is a $3-$modular lattice.
	\item[3.] Take $A=\QA{-2,-5}{\QQ}$ and $\ord$ with basis $\{i,\frac{1+i+j}{2},j,\frac{2+i+k}{4}\}$, $(\ord,q_1)$ is a $5-$modular lattice.
	\item[4.] Take $A=\QA{-3,-17}{\QQ}$ and $\ord$ with basis $\{1,\frac{1+i}{2},\frac{3+i+3j+k}{6},\frac{-3+i-2k}{6}\}$, $(\ord,q_1)$ is a $17-$modular lattice.
\end{enumerate}	
\end{example}
Note that the same construction for Examples 1 and 2 above appeared in \cite[p.266]{Martinet}.
%
%
%
\subsection{The case when $\ell$ is a positive integer}
Now we consider the case when $\ell$ is not necessarily square-free, i.e. $\ell$ being any positive integer.
Let $A$ be a totally definite quaternion algebra over $\QQ$ and let $\ord$ be any maximal order of $A$.
Let $r_p$ denote the exponent of prime $p$ in the prime factorization of $\ell$, i.e. $\ell=\prod_p p^{r_p}$.
If there exists an Arakelov-modular lattice of level $\ell$ over $\ord$, by Lemma \ref{lem:existenceQA} there exists $\beta\in\N{\ord}\cap\ord$ such that $\ell=\beta\bar{\beta}$.
And as in Eqs. (\ref{eqn:llam}) and (\ref{eqn:betalam}) we have
\[
	\ell\ord=\prod_{p|\ell,\P|p}\P^{r_pm_p},\ \ \beta\ord=\prod_{p|\ell,\P|p}\P^{\frac{r_pm_p}{2}}.
\]
We can see that if $m_p$ is odd, i.e. if $p\notin\text{Ram}_f(A)$, $r_p$ must be even.
Then
\[
	\beta\ord=\prod_{r_p\text{ even, }p\notin\text{Ram}_f(A),\P|p}\P^{\frac{r_p}{2}}\prod_{p|\ell,p\in\text{Ram}_f(A),\P|p}\P^{r_p},
\]
and
\[
	\D_{\ord/\ZZ}^{-1}(\beta\ord)=\prod_{r_p\text{ even, }p\notin\text{Ram}_f(A),\P|p}\P^{\frac{r_p}{2}}\prod_{p|\ell,p\in\text{Ram}_f(A),\P|p}\P^{r_p-1}\prod_{p\nmid\ell,p\in\text{Ram}_f(A),\P|p}\P^{-1}.
\]
As $\nrd{A/\QQ}{t}^{-1}\alpha^{-1}\in\QQ$, if $p\in\text{Ram}_f(A)$, $v_{\P}(\nrd{A/\QQ}{t}^{-1}\alpha^{-1})$ is even, thus we must have $\forall p\in\text{Ram}_f(A)$, $p|\ell$ and $r_p$ is odd.
\begin{proposition}
	Take a positive integer $\ell=\prod_pp^{r_p}$, there exists an Arakelov-modular lattice over $\ord$ iff the following conditions are all satisfied:
1. $\ell=\ell_1^2\ell_2$, where $\ell_2=\prod_{p\in\text{Ram}_f(A)}p^{r_p}$, $\ell_1$ is a positive integer coprime with $\ell_2$;

2. For all $p|\ell_2$, $r_p$ is odd;

3. There exists $\beta\in\N{\ord}\cap\ord$ such that $\ell=\beta\bar{\beta}$.	
\end{proposition}
\begin{proof}
In view of the above discussion, it suffices to prove if the conditions are satisfied, then there exists an Arakelov-modular lattice of level $\ell$.
We have
\[
	\D_{\ord/\ZZ}^{-1}(\beta\ord)=\prod_{p|\ell_1,\P|p}\P^{\frac{r_p}{2}}\prod_{p|\ell_2,\P|p}\P^{r_p-1}.
\]
Let $\alpha=\ell_1$, then
\[
	\alpha^{-1}\D_{\ord/\ZZ}^{-1}(\beta\ord)=\prod_{p|\ell_2,\P|p}\P^{r_p-1}.
\]
As $r_p$ are all odd for $p|\ell_2$, we can take
\[
	J = \prod_{p|\ell_2,\P|p}\P^{\frac{r_p-1}{2}}=\prod_{p\in\text{Ram}_f(A),\P|p}\P^{\frac{r_p-1}{2}}.
\]
Let $t=1$, $I=Jt$, then by Lemma \ref{lem:existenceQA}, $(I,q_\alpha)$ is an Arakelov-modular lattice of level $\ell$.
\end{proof}
\begin{remark}
	If $\ell$ is square-free, then $\ell_1=1$, $r_p=1$ for all $p|\ell_2$ and we get the same statement as in Proposition \ref{prop:QAoverQ}.	
\end{remark}
Since we are considering totally definite quaternion algebras $A$, Ram$_f(A)\neq\emptyset$.
Thus $\ell_2\neq1$, which implies
\begin{corollary}
There does not exist any Arakelov-modular lattice over $\ord$ of level $\ell$ for $\ell$ a square.
\end{corollary}
\begin{example}
	Take $\ell_1=1,\ell_2=2^3$, $A=\QA{-1,-1}{\QQ}$, so $\ell=8$ and Ram$_f(A)=\{2\}$.
	Let $\P$ be the ideal above $2$, then $(\P,q_1)$ is an Arakelov-modular lattice of level $8$ and dimension $4$.
	By Magma \cite{Magma}, this lattice is even with minimum $4$.
\end{example}
\begin{example}
	Take $\ell_1=1,\ell_2=3^3$, $A=\QA{-1,-3}{\QQ}$, so $\ell=27$ and Ram$_f(A)=\{3\}$.
	Let $\P$ be the ideal above $3$, then $(\P,q_1)$ is an Arakelov-modular lattice of level $27$ and dimension $4$.
	By Magma \cite{Magma}, this lattice is even with minimum $6$.
\end{example}
\begin{example}
	Take $\ell_1=2,\ell_2=3$, $A=\QA{-1,-3}{\QQ}$, so $\ell=12$ and Ram$_f(A)=\{3\}$.
	Then $(\ord,q_2)$ is an Arakelov-modular lattice of level $12$ and dimension $4$.
	By Magma \cite{Magma}, this lattice is even with minimum $4$.
\end{example}
The reader may wonder if there exist Arakelov-modular lattices when the base field $K$ is not $\QQ$.
We give such an example below.
\begin{example}
	Take $K=\QQ(\sqrt{6})$, $A=\QA{-1,-1}{K}$, $\ord$ with basis $\{1,1+i,1+j,1+i+j+k\}$.
	Then $(\P_2^{-1},1)$ is an Arakelov-modular lattice of level $6$.
	Here $\P_2$ is the unique prime ideal above the $\Oc_K-$ideal $\p_2$, where $\p_2\cap\ZZ=2\ZZ$.
\end{example}
The characterization of Arakelov-modular lattices over totally definite quaternion algebras over $K\neq\QQ$ may be an interesting topic for further research. 
%
%
%
\section*{Acknowledgments}
The author would like to thank Fr\'ed\'erique Oggier for her helpful advice.
This work is supported by Nanyang President Graduate Scholarship.
%
%
%


\begin{thebibliography}{99}

\bibitem{Batut} C. Batut, H.-G. Quebbemann, R. Scharlau, ``Computations of cyclotomic lattices'', \textit{Experimental Mathematics} \textbf{4} (1995), 175-179.

\bibitem{Bayer} E. Bayer-Fluckiger, I. Suarez, ``Modular lattices over Cyclotomic Fields'', \textit{Journal of Number Theory} \textbf{114} (2005), 394-411.

\bibitem{BayerIdeallattice} E. Bayer-Fluckiger, ``Ideal Lattices'', in A Panorama of Number Theory or The View from Baker’s Garden, edited by Gisbert Wustholz Cambridge Univ. Press, Cambridge (2002), 168–184.


\bibitem{Magma} W. Bosma, J. J. Cannon, C. Fieker, A. Steel (eds.), Handbook of Magma functions, Edition 2.22 (2016).
 

\bibitem{Conway} J.H. Conway, N.J.A. Sloane, ``Sphere packings, lattices and groups'', \textit{Springer}, New York, 1988.

\bibitem{Dummit} D. Dummit, R. Foote, ``Abstract Algebra Third edition'', \textit{John Wiley and Sons, Inc.}, Hoboken, 2004.

\bibitem{Ebeling} W. Ebeling, ``Lattices and codes: a course partially based on lecturers by F. Hirzebruch Advanced Lectures in Mathematics'', \textit{Springer}, Germany, 2013.


\bibitem{Feit} W. Feit, ``Some lattices over $\QQ{\sqrt{-3}}$'', \textit{Journal of Algebra} \textbf{52} (1978), 248-263.

\bibitem{Hardy} G.H. Hardy, E,M, Wright, ``An introduction to the theory of numbers'', \textit{Oxford University Press}, 1938.

\bibitem{Maclachlan} C. Maclachlan, A.W. Reid, ``The arithmetic of hyperbolic 3-manifolds'', Graduate Text in Math., \textit{Springer-Verlag}, Berlin, 2003.

\bibitem{Martinet} J. Martinet, ``Perfect lattices in Euclidean spaces'', \textit{Springer Science $\And$ Business Media}, 2013.

\bibitem{Neukirch} J. Neukirch, ``Algebraic Number Theory'', \textit{Springer-Verlag}, New York, 1999.

\bibitem{Quebbemann} H.-G. Quebbemann, ``Modular lattices in Euclidean Spaces'', \textit{Journal of Number Theory} \textbf{54} (1995), 190-202.



\bibitem{Reiner} I. Reiner, ``Maximal orders'',\textit{Academic Press}, New York, 1975.


\bibitem{LatticeWeb} N. J. A. Sloane and G. Nebe, ``Catalogue of Lattices'', published electronically at http://www.research.att.com/∼njas/lattices/.


\bibitem{Swinnerton-Dyer} H.P.F. Swinnerton-Dyer, ``A brief guide to Algebraic Number Theory'', \textit{Cambridge University Press}, 2001. 

\bibitem{Tu} F-T. Tu, Y. Yang, ``Lattice packing from quaternion algebras'', \textit{Algebraic Number Theory and Related Topics}, RIMS K\^oky\^uroku Bessatsu, \textbf{32}, (2012).

\bibitem{Voight} J. Voight, ``The arithmetic of quaternion algebras'', version April 2014, in preparation.

\bibitem{Vigneras} Marie-France Vign\'eas, ``Arithm\'éique des alg\`ebres quaternions'', \textit{Lecture Notes in Mathematics}, \textbf{800} (1980), Springer, Berlin.

\bibitem{Washington} L. Washington, ``Introduction to Cyclotomic Fields'', \textit{Springer-Verlag}, Berlin, 1982.


\end{thebibliography}
\end{document}